\documentclass[12pt]{article}
\RequirePackage[colorlinks,citecolor=blue,urlcolor=blue,linkcolor=blue]{hyperref}
\hypersetup{
colorlinks = true,
citecolor=blue,
urlcolor=blue,
linkcolor=blue,
pdfpagemode = UseNone
}
\usepackage[titletoc,title]{appendix}
\usepackage{graphicx,xspace,colortbl}
\usepackage{epstopdf}
\usepackage{amsmath,amsthm,amsfonts,amssymb}
\usepackage{color}
\usepackage{enumerate}
\usepackage{fancybox}
\usepackage{epsfig}
\usepackage{subfig}
\usepackage{pdfsync}
\usepackage{comment}
    \def\qed{\hfill$\sqcap\kern-8.0pt\hbox{$\sqcup$}$\\}
    \def\re{\textnormal {Re}}
    \def\im{\textnormal {Im}}

    \def\i{{\textnormal i}}

	\newtheorem{theorem}{Theorem}
	
	\newtheorem{proposition}{Proposition}
	
	\theoremstyle{definition}

	\newtheorem{remark}{Remark}

\title{On the Barnes double gamma function}
\author{ 
{Shahen Alexanian, Alexey Kuznetsov\footnote{Dept. of Mathematics and Statistics,  York University,
4700 Keele Street, Toronto, ON, M3J 1P3, Canada.   Email: akuznets@yorku.ca}}}

\date{\today}

\begin{document}
\maketitle

\begin{abstract}
We aim to achieve the following three goals. First of all, we collect all known definitions, transformation properties and functional identities of Barnes double gamma function $G(z;\tau)$. Second, we derive an algorithm for numerically computing the double gamma function and present its complete asymptotic expansion as $z\to \infty$. Third, we derive some new properties, including a new product identity and  new representations of the gamma modular forms $C(\tau)$ and $D(\tau)$. 
\end{abstract}
{\vskip 0.15cm}
{\vskip 0.25cm}


\section{Introduction}\label{section_Introduction}


The double gamma function $G(z;\tau)$ was introduced and thoroughly studied by Barnes in 1899   \cite{Barnes1899}. It seems that this function was studied previously by Alexeiewski and other researchers:  Barnes even calls it {\it Alexeiewski's function} in \cite{Barnes1899} and in \textsection 26 in \cite{Barnes1901}. The modern convention is to call $G(z;\tau)$ the double gamma function or the Barnes double gamma function. In the last three decades this special function has appeared more and more frequently in solutions to problems in probability \cite{Kuznetsov2011,KuznetsovPardo2013,Yor2009,Simon2018}, number theory \cite{Keating2000,Shintani1980} and physics \cite{Lawrie1994} and it has been a fertile object of study in its own right for its many interesting properties and connections to other special functions \cite{Michitomo_2001,RUIJSENAARS2000107}.

We will present two definitions of the double gamma function $G(z;\tau)$ as infinite products. First, for
$\tau \in {\mathbb C} \setminus (-\infty,0]$ we define $C(\tau)$ and $D(\tau)$ through the following   asymptotic expansions as $m\to +\infty$
\begin{align}
\label{def_C}
C(\tau)&:=\sum\limits_{k=1}^{m-1}\psi(k\tau)+\frac12 \psi(m\tau)-\frac{1}{\tau}\ln\left(\frac{\Gamma(m\tau)}{\sqrt{2\pi}}\right)-\frac{\tau}{12}\psi'(m\tau) \\ 
\nonumber
&+\frac{\tau^3}{720}\psi^{(3)}(m\tau)-\frac{\tau^5}{30240}
\psi^{(5)}(m\tau)+\frac{\tau^7}{1209600}
\psi^{(7)}(m\tau)+O(m^{-9}), \\
\label{def_D}
D(\tau)&:=\sum\limits_{k=1}^{m-1}\psi'(k\tau)+\frac12 \psi'(m\tau)-\frac{1}{\tau}\psi(m\tau)-\frac{\tau}{12}\psi''(m\tau) \\
\nonumber
&+\frac{\tau^3}{720}\psi^{(4)}(m\tau)-\frac{\tau^5}{30240}
\psi^{(6)}(m\tau)+\frac{\tau^7}{1209600}
\psi^{(8)}(m\tau)+O(m^{-10}).
\end{align}
Here $\psi(z)=\frac{d}{d z}\ln(\Gamma(z))$ is the digamma function (see \cite{Jeffrey2007}).  We would like to emphasize that  the above equations 
\eqref{def_C} and \eqref{def_D} serve two purposes: they give us definitions of $C(\tau)$ and $D(\tau)$ if we take a limit $m\to +\infty$ of the right-hand side of each equation  and they also provide a practical way to compute $C(\tau)$ and $D(\tau)$ if we take $m$ a large positive integer.
The asymptotic expansions \eqref{def_C} and \eqref{def_D} can be found on pages 362-363 in \cite{Barnes1899} (the higher order terms presented here were obtained by Euler-Maclaurin summation -- these are useful when computing $C(\tau)$ and $D(\tau)$ numerically). Barnes \cite{Barnes1899} calls these two functions $C(\tau)$
and $D(\tau)$ {\it the gamma modular forms}, and we will follow this convention.  Next, we define 
\begin{equation}\label{def_a_b}
a(\tau):=-\gamma \tau+\frac{\tau}2 \ln(2\pi \tau)+\frac12 \ln(\tau)-\tau C(\tau), \;\; 
b(\tau):=-\frac{\pi^2 \tau^2}{6} -\tau \ln(\tau)-\tau^2 D(\tau), 
\end{equation} where $\gamma=-\psi(1)=0.577215...$ is the Euler-Mascheroni constant.
Now we can define the double gamma function $G(z;\tau)$ (for $\tau \in {\mathbb C} \setminus (-\infty,0]$ and $z\in {\mathbb C}$) as an entire function of $z$ of order two having the Weierstrass infinite product representation: 
\begin{equation}\label{Weierstrass_G}
G(z;\tau):=\frac{z}{\tau} e^{a(\tau)\frac{z}{\tau}+b(\tau)\frac{z^2}{2\tau}} \prod\limits_{m\ge 0} \prod\limits_{n\ge 0} {}^{'}
\left(1+\frac{z}{m\tau+n} \right)e^{-\frac{z}{m\tau+n}+\frac{z^2}{2(m\tau+n)^2}},
\end{equation}
where the prime in the second product means that the term corresponding to $m=n=0$ is omitted. 
 According to the definition \eqref{Weierstrass_G}, it is clear that the function $G(z;\tau)$  has zeros on the lattice $m\tau+n$, $m\le 0$, $n\le 0$ (we note that the zeros are all simple if and only if $\tau$ is not rational). 

The second definition of $G(z;\tau)$ is through a single infinite product involving the classical gamma function: 
\begin{equation}\label{eq_G2_inf_prod_Gamma}
G(z;\tau)=\frac{1}{\tau \Gamma(z)} e^{\tilde a(\tau) \frac{z}{\tau}+\tilde b(\tau)\frac{z^2}{2\tau^2}} 
\prod\limits_{m\ge 1} \frac{\Gamma(m\tau)}{\Gamma(z+m\tau)} e^{z\psi(m\tau)+\frac{z^2}2 \psi'(m\tau)},
\end{equation}
where we denoted 
\begin{equation}\label{def_tilde_a_tilde_b}
\tilde a(\tau)=a(\tau)-\gamma \tau, \qquad
\tilde b(\tau)=b(\tau)+\frac{\pi^2 \tau^2}{6}.
\end{equation} 
It is easy to see that both infinite products in \eqref{Weierstrass_G} and 
\eqref{eq_G2_inf_prod_Gamma} converge absolutely. Barnes \cite{Barnes1899} proved the following crucial result:  the two definitions 
\eqref{Weierstrass_G} and \eqref{eq_G2_inf_prod_Gamma} are consistent and the function $G(z;\tau)$ satisfies the normalization condition $G(1;\tau)=1$ and the 
two functional equations 
\begin{equation}\label{eq:funct_rel_G}
G(z+1;\tau)=\Gamma\left(\frac{z}{\tau}\right) G(z;\tau), \;\;\; G(z+\tau;\tau)=(2\pi)^{\frac{\tau-1}2}\tau^{-z+\frac12}
\Gamma(z) G(z;\tau).
\end{equation} 
These last three properties can serve as the third definition of the Barnes double gamma function: it is a unique function that is entire in $z$ and analytic in $\tau \in {\mathbb C}\setminus (-\infty,0]$ and which satisfies the two functional equations \eqref{eq:funct_rel_G} and is normalized via $G(1;\tau)=1$. As we will see later, this third definition is very useful when proving various identities and transformation formulas involving $G(z;\tau)$. 

The double gamma function $G(z;\tau)$ is closely related to the symmetric version of the double gamma function $\Gamma_2(z;\omega_1,\omega_2)$, which  was studied by Barnes in \cite{Barnes1901}. This function is symmetric in $\omega_1$ and $\omega_2$, it satisfies the functional identities
\begin{align*}
\Gamma_2(z+\omega_1;\omega_1,\omega_2)&=
\sqrt{2\pi}\frac{ \omega_2^{\frac{1}{2}-\frac{z}{\omega_2}}}{\Gamma(
\frac{z}{\omega_2})}
\Gamma_2(z;\omega_1,\omega_2), \\
\Gamma_2(z+\omega_2;\omega_1,\omega_2)&=
\sqrt{2\pi}\frac{ \omega_1^{\frac{1}{2}-\frac{z}{\omega_1}}}{\Gamma(
\frac{z}{\omega_1})}
\Gamma_2(z;\omega_1,\omega_2)
\end{align*}
and the normalization condition 
$\Gamma_2(\omega_1;\omega_1,\omega_2)=\sqrt{2\pi /\omega_2}$. The function $\Gamma_2(z;\omega_1, \omega_2)$ can be expressed in terms of $G(z;\tau)$ as follows: 
\begin{equation}\label{G_as_Gamma1}
    \Gamma_2(z;\omega_1,\omega_2)=
    (2\pi)^{\frac{z}{2\omega_1}}
    \omega_2^{-\frac{z^2}{2\omega_1 \omega_2}+\frac{z(\omega_1+\omega_2)}{2\omega_1\omega_2}-1}
    G \Big(\frac{z}{\omega_1}; \frac{\omega_2}{\omega_1} \Big)^{-1}.
\end{equation}
This result is established in \textsection 26 in \cite{Barnes1901}.
To be more precise, the above four equations hold only when $\vert {\textnormal{arg}}(\omega_1)-{\textnormal{arg}}(\omega_2) \vert < \pi$, where the arguments of $\omega_i$ lie in $(-\pi, \pi)$. However, this case is all we need for our purposes and the reader can find the corresponding equations for the general case in \cite{Barnes1901}. 

The double gamma function enjoys a number of other important properties and transformation identities: 
\begin{itemize}
\item[(i)] A reflection formula: for $\im(\tau)>0$
\begin{equation}\label{eq_G_tau_q_Poch}
-2\pi \i \tau G\left(\frac12+z;\tau\right)G\left(\frac{1}{2}-z;-\tau\right)
=\frac{\left(-e^{2\pi \i z};q\right)_{\infty}}{(q;q)_{\infty}}, 
\end{equation}
where  $q:=e^{2 \pi \i \tau}$ and $(a;q)_{\infty}:=\prod_{n\ge 0} (1-aq^n)$ is the $q$-Pocchammer symbol. 
\item[(ii)] A modular transformation 
\begin{equation}\label{eq:G_1_over_tau}
G(z;\tau)=(2\pi)^{\frac{z}2 \left(1-\frac1{\tau} \right)} \tau^{\frac{z-z^2}{2\tau}+\frac{z}2-1}  G\left(\frac{z}{\tau};\frac{1}{\tau}\right).
\end{equation}
\item[(iii)] A multiplication formula: for any integers $p,q \in {\mathbb N}$
 \begin{align}
 \label{G_z_p/q_formula2}
 G(z;p \tau/q )&=
 q^{\frac{1}{2 p \tau}(z-1)(qz-p \tau) }  (2\pi)^{-\frac{q-1}{2}(z-1)}
 \\ \nonumber &\qquad \times 
 \prod\limits_{i=0}^{p-1} \prod\limits_{j=0}^{q-1} 
 \frac{G((z+i)/p+j \tau/q;\tau)}{G((1+i)/p+j \tau/q;\tau)}.
 \end{align}
 This formula has two important corollaries: setting $q=1$ and re-scaling $z\mapsto pz$ gives us
  \begin{equation}\label{G_z_p/q_formula3}
 G(pz;p \tau )=
 \prod\limits_{i=0}^{p-1} 
 \frac{G(z+i/p;\tau)}{G((1+i)/p;\tau)},
 \end{equation} 
 whereas if we set $q=p$ and re-scale $z\mapsto pz$ we get
 \begin{equation}\label{G_z_pz}
 G(pz;\tau )=
 p^{\frac{1}{2  \tau}(pz-1)(pz- \tau) }  (2\pi)^{-\frac{1}{2}(p-1)(pz-1)}
 \prod\limits_{i=0}^{p-1} \prod\limits_{j=0}^{p-1} 
 \frac{G(z+(i+j\tau)/p;\tau)}{G((1+i+j \tau)/p;\tau)}
 \end{equation}  
 \item[(iv)]  
  A product identity 
  \begin{equation}\label{eq:modular}
  G(z;\tau)=\Big(\frac{1+\tau}{\tau}\Big)^{\frac{z^2}{2\tau}-(1+\tau) \frac{z}{2\tau}+1}
  (2\pi)^{-\frac{z}{2\tau}}G(z+1;1+\tau)G(z/\tau;1+1/\tau).
 \end{equation}
 \item[(v)] 
 An integral representation: for $\re(z)>0$ and $\re(\tau)>0$
 \begin{align}
 \label{lnG_integral}
     \ln G(z;\tau) &= \int_0^{\infty} \bigg( \frac{e^{-\tau x}-e^{-z x}}{(1-e^{-x})(1-e^{-\tau x})}-z \frac{e^{-\tau x}}{1-e^{-\tau x}}
     \\ \nonumber & \qquad 
     +(z-1)\Big(\frac{z}{2\tau}-1\Big) e^{-\tau x} + \frac{e^{-x}}{1-e^{-x}} \bigg) \frac{dx}{x}.
 \end{align}
\end{itemize}

The reflection formula \eqref{eq_G_tau_q_Poch} was established in \textsection 13 in \cite{Barnes1899}. 
Identity \eqref{eq:G_1_over_tau} follows by setting $\omega_1=1$ and $\omega_2=\tau$ in \eqref{G_as_Gamma1} and using the fact that 
$\Gamma_2(z;\omega_1,\omega_2)=\Gamma_2(z;\omega_2,\omega_1)$. Formula 
\eqref{G_z_p/q_formula2} follows from \eqref{G_as_Gamma1} and the formula for $\Gamma_2(z;\omega_1/p,\omega_2/q)$ on page 359 in \cite{Barnes1901}. Formula \eqref{lnG_integral} can be found in \cite{Bill1997}[Section 5]. Identity \eqref{eq:modular} seems to be new: we present the proof of this formula in Section \ref{section_proofs}. In the same Section \ref{section_proofs} we also present direct proofs of 
\eqref{eq:G_1_over_tau} and \eqref{G_z_p/q_formula2}, which do not rely on the symmetric double gamma function $\Gamma_2(z;\omega_1,\omega_2)$. 

Our main goals in this paper are to develop an algorithm for numerical computation of $G(z;\tau)$ and to state a complete asymptotic expansion of this function as $z\to \infty$. In Section \ref{section_polynomials} we will introduce and study two sequences of polynomials which will be needed for achieving these goals in Section \ref{section_results}. The main results in Section \ref{section_results} are: (i) Theorem \ref{thm_main1}, which gives a formula that can be used to numerically compute $G(z;\tau)$; (ii)  Theorem \ref{thm2} and Proposition \ref{proposition_b0}, which state a complete asymptotic expansion of $G(z;\tau)$ as $z\to \infty$ and discuss the properties of the constant term in this expansion; (iii) Proposition \ref{Proposition_C_D_tau} where we present new formulas for the gamma modular forms $C(\tau)$ and $D(\tau)$.


\section{Two sequences of polynomials}\label{section_polynomials}


We recall that  the sequence $\{B_n\}_{n\ge 0}$ of the Bernoulli numbers is defined via the generating function
\begin{equation}\label{Bn_gen_function}
\frac{u}{e^u-1}=\sum\limits_{n\ge 0} B_n \frac{u^n}{n!}. 
\end{equation}
It is known that $B_n$ are all rational numbers, $B_{2m+1}=0$ for $m\ge 1$ and the first six non-zero  values are 
$$
B_0=1, \; B_1=-\frac{1}{2}, \; B_2=\frac{1}{6}, \; B_4=-\frac{1}{30}, \; B_6=\frac{1}{42}, 
\; B_8=-\frac{1}{30}, \dots 
$$
The Bernoulli polynomials are defined as
$$
B_n(x)=\sum\limits_{k=0}^n \binom{n}{k} B_{n-k} x^k
$$
and have the generating function
\begin{equation}\label{Bn_pols_gen_function}
    \frac{u e^{xu}}{e^u-1}=\sum\limits_{n\ge 0} B_n(x) \frac{u^n}{n!}. 
\end{equation}

Now we define the first sequence of polynomials that will be important for us later: 
\begin{equation}\label{def_qn}
q_n(\tau):=\sum\limits_{k=0}^n \binom{n}{k} B_k B_{n-k} \tau^k, \;\;\; n\ge 0. 
\end{equation}
If $n$ is even then $q_{n}(\tau)$ is a polynomial of degree $n$. 
If $n=2m+1\ge 3$ is odd then the only non-zero terms in the above sum are those with $k=1$ and $k=n-1$, so that in this case $q_n(\tau)$ is a polynomial of degree $n-1$ given by
\begin{equation}
q_{2m+1}(\tau)=-(m+\tfrac{1}{2}) B_{2m} \times  \tau (1+\tau^{2m-1}). 
\end{equation}
The list of polynomials $q_n(\tau)$ for $0\le n \le 21$ is presented in Appendix \ref{AppendixA}. In the next result we collect some properties of polynomials $q_n(\tau)$. 

\newpage 
\begin{proposition}\label{proposition1}
${}$
\begin{itemize}
\item[(i)] The exponential generating functions of polynomials $\{q_n(\tau)\}_{n\ge 0}$ is
\begin{equation}\label{qn_generating_function}
\sum\limits_{n\ge 0} 
q_{n}(\tau)\frac{u^n}{n!} =\frac{\tau u^2}{(e^{u}-1)(e^{\tau u}-1)}.
\end{equation}
\item[(ii)] For $n\ge 0$ we have $q_n(\tau)=\tau^n q_n(\tau^{-1})$.

\item[(iii)] The polynomials $q_n(\tau)$ can be computed recursively via 
\begin{equation}
\label{qn_recursion}
q_n(\tau)=-\frac{1}{\tau} \sum\limits_{k=1}^n \binom{n}{k} \frac{(1+\tau)^{k+2}-1-\tau^{k+2}}{(k+1)(k+2)}
q_{n-k}(\tau)
\end{equation}
by starting with $q_0(\tau)=1$. 
\item[(iv)] The polynomials $q_n(\tau)$ satisfy the following summation identities: for $n\ge 0$, $\tau \in {\mathbb C}$ and $y \in {\mathbb C}\setminus\{0\}$
\begin{align}
\label{sum_identity1}
&\sum\limits_{k=0}^n \binom{n}{k} q_k(\tau)=(-1)^n  q_n(-\tau), \\
\label{sum_identity2}
&\sum\limits_{k=0}^n \binom{n+1}{k} q_k(\tau)=(n+1) B_n \tau^n, \\
\label{sum_identity3}
&
\sum\limits_{k=0}^n \binom{n+1}{k+1}   
(B_{k+1}(y)-B_{k+1}) y^{n-k} q_{n-k}(\tau)=(n+1)y  q_n(\tau y),\\ 
\label{sum_identity4}
&\sum\limits_{k=0}^n \binom{n+1}{k+1}   
(B_{k+1}(y)-B_{k+1}) \Big(\frac{\tau}{y}\Big)^k q_{n-k}(\tau)=(n+1) y q_n\Big(\frac{\tau}{y}\Big).
\end{align}
\end{itemize}
\end{proposition}
\begin{proof}
Formula \eqref{qn_generating_function} follows from \eqref{Bn_gen_function} and  \eqref{def_qn}. The identity $q_n(\tau)=\tau^n q_n(\tau^{-1})$ from item (ii) follows at once from the definition \eqref{def_qn}. To prove \eqref{qn_recursion} we first derive a series expansion 
\begin{align*}
\frac{1}{\tau u^2} (e^u-1)(e^{\tau u}-1)&=\frac{1}{\tau u^2}
(e^{(1+\tau)u}-e^{u}-e^{\tau u}+1)\\
&=\frac{1}{\tau} \sum\limits_{n\ge 0} \big( (1+\tau)^{n+2}-1-\tau^{n+2} \big) \frac{u^n}{(n+2)!}.
\end{align*}
Using the above series and \eqref{qn_generating_function} we obtain
$$
1=\sum\limits_{n\ge 0} 
q_{n}(\tau)\frac{u^n}{n!} \times \frac{1}{\tau} \sum\limits_{n\ge 0} \big( (1+\tau)^{n+2}-1-\tau^{n+2} \big) \frac{u^n}{(n+2)!},
$$
and inspecting the coefficient in front of $u^n$ (with $n\ge 1$) we conclude that 
$$
\frac{1}{\tau} \sum\limits_{k=0}^n \binom{n}{k} \frac{(1+\tau)^{k+2}-1-\tau^{k+2}}{(k+1)(k+2)}
q_{n-k}(\tau)=0,
$$
which is equivalent to \eqref{qn_recursion}.

Identity \eqref{sum_identity1} is established by comparing the coefficients in front of $u^n$ in the identity
\begin{align*}
    \sum\limits_{n\ge 0} (-1)^n q_{n}(-\tau) \frac{u^n}{n!}&=\frac{\tau u^2}{(1-e^{-u})(e^{\tau u}-1)}\\
    &=
    e^u \times \frac{\tau u^2}{(e^{u}-1)(e^{\tau u}-1)} =\sum\limits_{n\ge 0} \frac{u^n}{n!} 
    \times \sum\limits_{n\ge 0} 
q_{n}(\tau) \frac{u^n}{n!}.
\end{align*}
To prove \eqref{sum_identity2} we expand each term in square brackets in 
$$
 \Big[ \frac{e^u-1}{u} \Big] \times \Big[ \frac{\tau u^2}{(e^u-1)(e^{\tau u}-1)}\Big]=
 \Big[ \frac{\tau u}{e^{\tau u}-1} \Big]
$$
in power series in $u$ (with the help of Taylor series for $e^u$ and formulas 
\eqref{Bn_gen_function} and \eqref{qn_generating_function}) and compare the coefficients in front of $u^n$. Similarly, formula \eqref{sum_identity3} is obtained by expanding each term in square brackets in
the identity 
$$
\Big[\frac{e^{yu}-1}{e^u-1}\Big]\times \Big[ \frac{\tau (yu)^2}{(e^{yu}-1)(e^{\tau (y u)}-1)}\Big]=
y\times \Big[\frac{(y\tau) u^2}{(e^{u}-1)(e^{(y\tau) u}-1)}\Big]
$$
and comparing the coefficients in front of $u^n$.
One would also need to use the following series expansion
$$
\sum\limits_{n\ge 0} \frac{B_{n+1}(y)-B_{n+1}}{n+1} \times \frac{u^n}{n!}=
\frac{e^{yu}-1}{e^u-1},
$$
which is easily obtained from \eqref{Bn_gen_function} and \eqref{Bn_pols_gen_function}. 
Finally, formula \eqref{sum_identity4} follows from \eqref{sum_identity3} and identity 
$q_n(\tau)=\tau^n q_n(\tau^{-1})$. 
\end{proof}

To state our main results in Section \ref{section_results}, we will also need a sequence of polynomials 
$\{P_n(z;\tau)\}_{n\ge 1}$, which are  defined as follows: 
\begin{equation}\label{def_P_k}
P_n(z;\tau):=\sum\limits_{k=1}^n 
\binom{n+2}{k+2}   q_{n-k}(\tau) z^{k-1}. 
\end{equation}
It is clear that $P_n(z;\tau)$ is a polynomial with rational coefficients; it is a monic polynomial of degree $n-1$ in $z$ and a polynomial of degree not greater than $n-1$ in $\tau$. The first five polynomials $P_n(z;\tau)$ are 
\begin{align*}
P_1(z;\tau)&=1, \\
P_2(z;\tau)&=z-2(\tau+1), \\
P_3(z;\tau)&=z^2 -\frac{1}{6}\big(15 (\tau+1) z - 10 (\tau^2 +3 \tau + 1)\big),\\
P_4(z;\tau)&=z^3-\frac{1}{2}\big(6 (\tau+1) z^2 - 5 (\tau^2+3\tau+1)z+10\tau(\tau+1)\big), \\
P_5(z;\tau)&=z^4-\frac{1}{12} \big(42 (\tau+1)z^3-42(\tau^2+3 \tau+1) z^2\\
& \qquad \qquad +105\tau (\tau+1) z 
+14 (\tau^4-5 \tau^2 +1)\big).
\end{align*}
Properties of polynomials $P_n(z;\tau)$ are presented in the next result. 

\newpage

\begin{proposition}\label{proposition2}
${}$
\begin{itemize}
    \item[(i)] The sequence of polynomials 
    $\{P_n(z;\tau)\}_{n\ge 1}$ has generating function
    \begin{equation}\label{P_k_exp_generating_function}
\sum\limits_{n\ge 1}  P_{n}(z;\tau) \frac{u^{n}}{(n+2)!} =
\frac{\tau}{z^3} \times 
\frac{e^{uz}-1-uz-\frac{u^2z^2}{2}}{(e^{ u}-1)(e^{\tau u}-1)}.
\end{equation}
\item[(ii)] For $n\ge 1$ we have $P_n(z/\tau;1/\tau)=\tau^{1-n} P_n(z;\tau)$. 
    \item[(iii)] An alternative expression for $P_n(z;\tau)$ is
\begin{equation}\label{def_P_k2}
P_n(z;\tau)=\sum\limits_{k=1}^n 
\binom{n+2}{k+2}  \tau^{n-k} B_{n-k} \tilde B_{k+2}(z), 
\end{equation}
where $\tilde B_k(z)$ are polynomials of degree $k-3$ defined by
$$
 \tilde B_k(z)=z^{-3} \Big[B_{k}(z)-B_{k}(0)- B_{k}'(0) z-B_{k}''(0)\frac{z^2}{2}  \Big], \;\;\; k\ge 3.  
$$
\item[(iv)] For $n\ge 2$ the polynomials $P_n(z;\tau)$ can be computed recursively via
\begin{align}
\label{Pk_recursion}
P_n(z;\tau)=
z^{n-1}-\frac{1}{\tau} \sum\limits_{k=1}^{n-1} \binom{n+2}{k+2} \frac{(1+\tau)^{k+2}-1-\tau^{k+2}}{(n-k+1)(n-k+2)}
P_{n-k}(z;\tau),
\end{align}
by starting with $P_1(z;\tau)=1$. 
\end{itemize}
\end{proposition}
\begin{proof}
To prove \eqref{P_k_exp_generating_function} 
we use series expansions \eqref{qn_generating_function} and
$$
\frac{e^{uz}-1-uz-\frac{u^2 z^2}{2}}{u^2 z^3}
=\sum\limits_{n\ge 1} \frac{z^{n-1}u^n}{(n+2)!}
$$
and check that the coefficient in front of $u^n$ in the series expansion of 
$$
\Big[ \frac{e^{uz}-1-uz-\frac{u^2 z^2}{2}}{u^2 z^3} \Big] \times 
\Big[ \frac{\tau u^2}{(e^{ u}-1)(e^{\tau u}-1)}\Big]
$$
is equal to $P_n(z;\tau)/(n+2)!$. Identity $P_n(z/\tau;1/\tau)=\tau^{1-n} P_n(z;\tau)$ follows from the definition of $P_n(z;\tau)$  (formula \eqref{def_P_k})
and result in item (ii) in Proposition \ref{proposition1}. Formula \eqref{def_P_k2} is easily derived from 
\eqref{P_k_exp_generating_function} and series expansion
\begin{align*}
\frac{e^{uz}-1-uz-\frac{u^2z^2}{2}}{z^3 u(e^{u}-1)}=
\sum\limits_{n\ge 1}
\tilde B_{n+2}(z)
\times \frac{u^n}{(n+2)!}, 
\end{align*}
which follows from \eqref{Bn_pols_gen_function}. Finally, \eqref{Pk_recursion} is proved in exactly the same way as formula \eqref{qn_recursion} in Proposition 
\ref{proposition1} and we leave the details to the reader. 
\end{proof}


\section{Main results}\label{section_results}


First we will address the question of how to numerically  compute the Barnes double gamma function $G(z;\tau)$. If $z=m\tau+n$ with either $m$ or $n$ positive, then $G(z;\tau)$ can be computed via functional equations \eqref{eq:funct_rel_G} and the fact that $G(1;\tau)=1$. On the other hand, if $\tau=q/p$ is a rational number, then the problem of computing $G(z;\tau)$ is reduced with the help of \eqref{G_z_p/q_formula2} to the problem of numerically evaluating $G(z)=G(z;1)$, which can be done efficiently by an algorithm presented in \cite{Kuznetsov2022}. Here we present an algorithm for computing values of $G(z;\tau)$ for complex values of $z$ and/or $\tau$, which only requires being able to evaluate logarithm of the gamma function in the entire complex plane (for this one could use Lanczos approximation \cite{Lanczos1964} or approximations developed in 
\cite{Kuznetsov2022}.)

Formula \eqref{eq_G2_inf_prod_Gamma} represents the double gamma function $G(z;\tau)$ as an infinite product involving gamma functions. We take the truncated part of this infinite product as an approximation of the Barnes double gamma function
\begin{equation}
\label{def_G_N}
G_N(z;\tau):=\frac{1}{\tau \Gamma(z)} e^{\tilde a(\tau) \frac{z}{\tau}+\tilde b(\tau)\frac{z^2}{2\tau^2}} 
\prod\limits_{m=1}^{N} \frac{\Gamma(m\tau)}{\Gamma(z+m\tau)} e^{z\psi(m\tau)+\frac{z^2}2 \psi'(m\tau)}
\end{equation}
where $\tilde a(\tau)$ and $\tilde b(\tau)$ are defined in \eqref{def_tilde_a_tilde_b}. In the next result we present the complete asymptotic expansion in powers of $N^{-1}$ of the error term of this approximation. 

\begin{theorem}\label{thm_main1} Assume that $z\in {\mathbb C}$
and $|\arg(\tau)|<\pi$. Then for any $M=1,2,3,\dots$  we have  
\begin{equation}\label{eqn:compute_G}
G(z;\tau)=G_N(z;\tau) \exp\Big( z^3 \sum \limits_{k=1}^M \frac{(- \tau)^{-k-1}P_k(z;-\tau)}{k(k+1)(k+2)} \times N^{-k}
+O(N^{-M-1})\Big), 
\end{equation}
as $N\to +\infty$.
\end{theorem}
\begin{proof}
We fix the values of $z$ and $\tau$ and denote 
\begin{equation}\label{def:r_m}
r_m(z;\tau):=\ln \Gamma(m\tau)-\ln \Gamma(z+m\tau) + z\psi(m\tau)+\frac{z^2}2 \psi'(m\tau)
\end{equation}
and
\begin{align}
\label{def_R_N}
R_N(z;\tau):=
\sum\limits_{m=N+1}^{\infty} r_{m}(z;\tau). 
\end{align}
We assume that $N$ is big enough, so that the circle 
$\{w \in {\mathbb C} \; : \; |w-N \tau|\le 2|z|\}$ does not intersect the negative half-line $(-\infty, 0]$. Since $\tau \in {\mathbb C} \setminus (-\infty, 0]$, this condition implies that for $m>N$ we have $z+m \tau \notin (-\infty,0]$ so that all the terms in \eqref{def_R_N} are well defined. 
It is clear from \eqref{eq_G2_inf_prod_Gamma} that
\begin{equation}\label{G_G_N_formula}
G(z;\tau)=G_N(z;\tau) \times \exp\big( R_N(z;\tau)\big),
\end{equation}
thus to prove \eqref{eqn:compute_G} we need to find an asymptotic expansion of $R_N(z;\tau)$ in powers of $N^{-1}$. 

Our first step is to write the Taylor expansion with the remainder term
\begin{align*}
\ln \Gamma(z + m \tau)&=\ln \Gamma(m \tau)+ z \psi( m \tau)\\
&
+ \frac{z^2}{2!} \psi'(m \tau)+\dots+ \frac{z^{M+2}}{(M+2)!} \psi^{(M+1)}(m \tau)+{\epsilon}_{M+2}(z,m\tau). 
\end{align*}
From the above series and \eqref{def:r_m} we  obtain 
$$
r_m(z;\tau)=-\sum\limits_{l=1}^M \frac{z^{l+2}}{(l+2)!} \psi^{(l+1)}(m \tau)-{\epsilon}_{M+2}(z,m\tau). 
$$
 Using the Taylor polynomial approximation with the error term in the form
$$
f(z)=\sum\limits_{k=0}^n \frac{z^k}{k!} f^{(k)}(0) 
 +\frac{z^{n+1}}{n!} 
\int_0^1 f^{(n+1)}(sz) (1-s)^n d s, 
$$
with $f(z)=\ln \Gamma(z+m\tau)$ and $n=M+2$, coupled with the maximum modulus principle, we obtain an upper bound
$$
\vert {\epsilon}_{M+2}(z,m\tau) \vert \le \frac{\vert z \vert^{M+3}}{(M+3)!} \times \max\limits_{|w|=|z|}
\vert \psi^{(M+2)}(w+ m \tau) \vert 
$$
It is known (see formula 5.15.9 in \cite{NIST:DLMF}) that $\psi^{(l+1)}(z)$ has asymptotic expansion
$$
\psi^{(l+1)}(z) \sim \sum\limits_{j\ge 0}  \frac{(j+l)!}{(j)!}\frac{(-1)^{j+l} B_{j} }{z^{l+j+1}}, \;\;\; z\to \infty. 
$$
Moreover, this asymptotic expansion is uniform in domains of the form
$\vert\arg(z)\vert<\pi-\delta$, for arbitrary small $\delta>0$. 
Therefore we have an estimate
${\epsilon}_{M+2}(z,m\tau)=O(m^{-M-2})$
and we can write down an asymptotic expansion of $r_m(z;\tau)$ in powers of $m^{-1}$: 
\begin{align*}
r_m(z;\tau)&=
-\sum\limits_{l=1}^M \frac{z^{l+2}}{(l+2)!} \psi^{(l+1)}(m \tau)+O(m^{-M-2})\\
&
=-\sum\limits_{l=1}^M \frac{z^{l+2}}{(l+2)!} 
\times \bigg\{\sum\limits_{j=0}^{M}  \frac{(j+l)!(-1)^{j+l} B_{j} }{(j)!(m\tau)^{l+j+1}}\bigg\}+O(m^{-M-2}).
\end{align*}
On the other hand, using the Euler-Maclaurin formula we establish the following asymptotic expansion, valid for any $s>1$:
\begin{align*}
\sum\limits_{m=N+1}^{\infty} m^{-s} \sim 
\sum\limits_{k\ge 0} \frac{(s)_{k-1}}{(k)!} B_{k} N^{1-s-k}, \;\;\; N\to +\infty. 
\end{align*}
Here $(s)_k:=s(s+1)\dots (s+k-1)$ is the Pochhammer symbol. 
From the above two formulas and \eqref{def_R_N} we obtain
an asymptotic expansion of $R_N(z;\tau)$ in powers of $N^{-1}$:
\begin{align*}
R_N(z;\tau)&=\sum\limits_{m=N+1}^{\infty} r_{m}(z;\tau) =-\sum\limits_{l=1}^M \frac{z^{l+2}}{(l+2)!} 
\\ & 
\times \bigg\{\sum\limits_{j=0}^{M}  \frac{(-1)^{j+l} B_{j} (j+l)!}{(j)!\tau^{l+j+1}}
\times \sum\limits_{k=0}^M \frac{(l+j+1)_{k-1}}{(k)!} B_{k} N^{-j-k-l}\bigg\}+O(N^{-M-1}).
\end{align*}
Next we note that $(j+l)!(l+j+1)_{k-1}=(l+j+k-1)!$, we introduce a new variable of summation $i=j+k+l$ and collect the terms in front of $N^{-i}$ (for $i=1,2,\dots,M$) to obtain
\begin{align*}
R_N(z;\tau)&=\sum\limits_{i=1}^M 
N^{-i} (i-1)! (-\tau)^{-i-1} \sum\limits_{l=1}^i 
\frac{z^{l+2}}{(l+2)!} 
\sum\limits_{k=0}^{i-l}  \frac{ (-\tau)^{k} B_{k}B_{i-l-k}}{(k)!
(i-l-k)!}
  +O(N^{-M-1})\\
  &=
  z^3\sum\limits_{i=1}^M 
N^{-i} (i-1)! (-\tau)^{-i-1} \sum\limits_{l=1}^i 
\frac{z^{l-1}}{(l+2)!(i-l)!} q_{i-l}(-\tau) 
  +O(N^{-M-1})\\
  &=z^3 \sum \limits_{i=1}^M \frac{(- \tau)^{-i-1}P_i(z;-\tau)}{i(i+1)(i+2)} \times N^{-i}
+O(N^{-M-1}).
\end{align*}
The above result and \eqref{G_G_N_formula} give us the desired formula \eqref{eqn:compute_G}. 
\end{proof}

To illustrate how Theorem \ref{thm_main1} may be used to numerically compute  the values of $G(z;\tau)$, we perform the following experiments. We set $\tau=\sqrt{3}$ and compute $C(\tau)$ and $D(\tau)$ via \eqref{def_C} and \eqref{def_D} with $m=1000$. For values  $z\in \{1, \sqrt{3}\}$ we compute the values of $\{P_n(z;\tau)\}_{1\le n \le 10}$ via recursion \eqref{Pk_recursion}. Then we set $M=10$ and $N=1000$ in \eqref{eqn:compute_G} and compute the values 
\begin{align*}
    G(1;\sqrt{3})&=0.9999999999999999999999999999992..., \\
    G(\sqrt{3};\sqrt{3})&=1.4889283353650864545337314811487....
\end{align*}
To perform these computations we used Matlab and variable-precision arithmetic functions (with a working accuracy of 50 digits).  These values are correct to within $10^{-30}$, as can be checked via the exact evaluations $G(1;\tau)=1$ and 
 \begin{equation}
 \label{eqn:G(tau;tau)}
 G(\tau;\tau)=(2\pi)^{(\tau-1)/2}\tau^{-1/2},
 \end{equation}
 which follows from \eqref{eq:G_1_over_tau}.  

Our second goal is to discuss the asymptotic expansion of the double gamma function $G(z;\tau)$ as $z\to \infty$ (for fixed $\tau$). 
Billingham and King \cite{Bill1997}  found the first five terms of the asymptotic expansion (they also considered asymptotic expansions for small and large values of $\tau$ and for small values of $z$). Their result can be stated as follows:  as $z\to \infty$  we have
\begin{align}\label{eqn_lnG_E}
\ln G(z;\tau)&=(a_2(\tau) z^2+a_1(\tau) z + a_0(\tau)) \ln(z)\\
\nonumber
&+b_2(\tau) z^2+ b_1(\tau) z + b_0(\tau)
+ {\mathcal E}(z;\tau),
\end{align}
where 
\begin{align}
\nonumber
a_0(\tau)&=\frac{\tau}{12}+\frac{1}{4}+\frac{1}{12 \tau}, \\
\nonumber
a_1(\tau)&=-\frac{1}{2} \Big( 1 + \frac{1}{\tau} \Big),\\
\label{eqn_a0_a1_a2_b0_b1}
a_2(\tau)&=\frac{1}{2\tau}, \\
\nonumber
b_1(\tau)&=\frac{1}{2} \Big(\Big(\frac{1}{\tau}+1\Big)(1+\ln(\tau))+\ln(2\pi) \Big),\\
\nonumber
b_2(\tau)&=-\frac{1}{2\tau} \Big( \frac{3}{2}+ \ln(\tau)\Big),
\end{align}
and ${\mathcal E}(z;\tau)=O(1/z)$. The domain where this asymptotic expansion holds is not stated explicitly in 
\cite{Bill1997}, as the authors were mostly interested in real values of $z$ and $\tau$, however from the derivation of this asymptotic result one can conclude that it holds for fixed $\tau \in {\mathbb C} \setminus (-\infty,0]$ and large values of $z$ in the domain ${\mathbb C}\setminus \{x+\tau y \, : \, x,y\le 0\}$ (the complex plane minus the cone containing the roots of the double gamma function $G(z;\tau)$). The value of $b_0(\tau)$ was not found explicitly in   \cite{Bill1997}: as we show in the next result, $b_0(\tau)$ is  a rather complicated transcendental function, unlike the other coefficients $a_i(\tau)$ and $b_i(\tau)$.

\begin{proposition}\label{proposition_b0}
 Assume that $\tau \in {\mathbb C}\setminus(-\infty,0]$. 
\begin{itemize}
\item[(i)] The function $b_0(\tau)$ is equal to
\begin{align}\label{formula_b0}
    b_0(\tau)&=\int_0^1  \int_0^1 \ln G(x + \tau y; \tau) d x d y - a_0(\tau) \ln(\tau)-\frac{1}{4}(\tau+1) \ln(2\pi) \\
    \nonumber
    &=\frac{1}{3} \Big\{ \ln \Big[G\Big(\frac{1}{2};\tau\Big)^2 G(\tau;2\tau )    \Big] -
\frac{1+\tau}{2}\ln(2\pi)-a_0(\tau) \ln(\tau^3/2) - \ln(2) \Big\}.
\end{align}
It can be explicitly evaluated at $\tau=1$: 
\begin{equation}\label{eqn_b0_1}
b_0(1)=\frac{1}{12}-\ln(A)-\frac{1}{2} \ln (2\pi), 
\end{equation}
where $A=1.282427129...$  is the Glaisher-Kinkelin constant. 
\item[(ii)] The function $b_0(\tau)$ satisfies the following identities: 
\begin{align}
\label{b0_identity1}
    b_0(\tau^{-1})&=b_0(\tau)+\frac{\ln(\tau)}{12 \tau} (1+15 \tau + \tau^2),\\
    \nonumber
    b_0(\tau)&=b_0(1+\tau)+b_0(1+1/\tau)
    \\ \label{b0_identity2} &+\frac{1}{2}\ln ( 
2\pi (1+\tau)^3) - \Big(17+\frac{1}{\tau(1+\tau)}\Big) \ln(\tau), \\
\nonumber
\label{b0_identity3}
b_0(p \tau/q)&=pq \Big(b_0(\tau)+a_0(\tau) \ln(\tau)\Big)-a_0(p \tau/q) \ln(p \tau) \\&+\Big(p-1+(q-1)(p(\tau+1)+1)\Big) \frac{1}{4} \log(2\pi) + \frac{1}{2}\ln(q)
\\
\nonumber 
&\qquad -\sum\limits_{i=0}^{p-1} \sum\limits_{j=0}^{q-1} \ln G((1+i)/p+j\tau/q;\tau), \;\;\; p,q \in {\mathbb N}.
\end{align}
\end{itemize}
\end{proposition}
\begin{proof}
Let us prove the results in item (ii). 
To prove formula \eqref{b0_identity1} we take logarithm of both sides of  \eqref{eq:G_1_over_tau} and consider the first six terms of the asymptotic expansion as $z\to \infty$ from 
\eqref{eqn_lnG_E}: 
\begin{align*}
    &(a_2(\tau) z^2+a_1(\tau) z + a_0(\tau)) \ln(z)+b_2(\tau) z^2+ b_1(\tau) z + b_0(\tau)+O(1/z)
    \\&=
    \frac{z}{2} \Big( 1-\frac{1}{\tau}\Big) \ln(2\pi) + \Big(\frac{z-z^2}{2\tau}+\frac{z}{2}-1\Big) \ln(\tau) \\
    &\qquad \qquad+
       \Big(a_2(1/\tau) (z/\tau)^2+a_1(1/\tau) (z/\tau) + a_0(1/\tau)\Big) \ln(z/\tau)\\
       &\qquad \qquad+b_2(1/\tau) (z/\tau)^2+ b_1(1/\tau) (z/\tau) + b_0(1/\tau)+O(1/z). 
\end{align*}
Using formulas \eqref{eqn_a0_a1_a2_b0_b1} the above equation can be simplified to 
$$
b_0(\tau)+O(1/z)=-\frac{\ln(\tau)}{12 \tau} (1+15 \tau + \tau^2)+
b_0(1/\tau)+O(1/z). 
$$
This simplification can be checked with symbolic computation software. The above equation implies \eqref{b0_identity1}.

Formula \eqref{b0_identity2} is established in the same way. 
Using the first functional equation in \eqref{eq:funct_rel_G} we rewrite \eqref{eq:modular} in the form
  \begin{equation*}
  G(z;\tau)=\Big(\frac{1+\tau}{\tau}\Big)^{\frac{z^2}{2\tau}-(1+\tau) \frac{z}{2\tau}+1}
  (2\pi)^{-\frac{z}{2\tau}}\Gamma\Big(\frac{z}{1+\tau}\Big)G(z;1+\tau)G(z/\tau;1+1/\tau). 
 \end{equation*}
Taking logarithms of both sides and considering the constant term in the asymptotic expansion as $z\to \infty$ (and using Stirling's asymptotic formula for $\ln \Gamma(z/(1+\tau))$ we  obtain \eqref{b0_identity2}. 

To prove formula \eqref{b0_identity3} we will compute the constant term in the asymptotic expansion as $z\to \infty$ of the logarithm of the right-hand side in \eqref{G_z_p/q_formula2}. 
Let $\alpha$ and $\beta$ be positive numbers. According to \eqref{eqn_lnG_E}
\begin{align*}
     \ln G(z/\alpha+\beta;\tau)&=
     \Big(a_2(\tau) (z/\alpha+\beta)^2 +a_1(\tau) (z/\alpha+\beta) + a_0(\tau)\Big) \ln(z/\alpha+\beta)\\
&+b_2(\tau) (z/\alpha+\beta)^2+ b_1(\tau) (z/\alpha+\beta) + b_0(\tau)+O(1/z).
\end{align*}
Using the Taylor series for the logarithm function we find 
\begin{align*}
    \ln(z/\alpha+\beta)&= \ln (z) - \ln(\alpha) + \ln(1+\alpha \beta/z)\\
     &=\ln (z) - \ln(\alpha) + \frac{\alpha \beta}{z}-\frac{(\alpha \beta)^2}{2z^2}+O(1/z^3) 
\end{align*}
Combining the above two formulas and simplifying the result  we find that the constant term in the asymptotic expansion of 
$\ln G(z/a+b;\tau)$ as $z\to \infty$ is 
\begin{align*}
  &  {\mathcal C}(\alpha,\beta;\tau):=-(a_2(\tau) \beta^2 +a_1(\tau) \beta + a_0(\tau))\ln(\alpha)\\
  &+2 a_2(\tau) \beta^2+ a_1(\tau)\beta-a_2(\tau) \beta^2/2+b_2(\tau) \beta^2 +b_1(\tau) \beta + b_0(\tau).
\end{align*}
Using the above formula and the identities 
\begin{align*}
    \sum\limits_{i=0}^{p-1} \sum\limits_{j=0}^{q-1} 
    (i/p+ j \tau/q)&=\frac{1}{2} ( q(p-1) + \tau p (q-1)), \\
    \sum\limits_{i=0}^{p-1} \sum\limits_{j=0}^{q-1} 
    (i/p+ j \tau/q)^2&=\frac{q (p-1) (2p-1)}{6 p}+ \frac{\tau}{2} (p-1)(q-1)
    +\frac{\tau^2 p (q-1)(2q-1)}{6q}
\end{align*}
we compute 
\begin{align*}
    \sum\limits_{i=0}^{p-1} \sum\limits_{j=0}^{q-1} C(p, i/p+j \tau/q; \tau)&=
    pq(b_0(\tau)+a_0(\tau)\ln(\tau))-a_0(p\tau/q)\ln(p \tau)\\
    &+(p-1+(q-1)(p(\tau+1)-1))\ln(2\pi)/4.
\end{align*}
This last computation is easy to verify by symbolic computation software. 
Next, we find that the constant term in the asymptotics (as $z\to \infty$) of
$$
\ln \Big[ p^{\frac{1}{2  \tau}(pz-1)(pz- \tau) }  (2\pi)^{-\frac{1}{2}(p-1)(pz-1)} \Big]
$$
is $\frac{1}{2} (q-1) \ln(2\pi)
+\frac{1}{2} \ln(q)$. Combining the above two formulas we find that the constant term in the asymptotic expansion of the logarithm of the right-hand side of  
\eqref{G_z_p/q_formula2} is 
\begin{align*}
  &  pq \Big(b_0(\tau)+a_0(\tau) \ln(\tau)\Big)-a_0(p \tau/q) \ln(p \tau) +\Big(p-1+(q-1)(p(\tau+1)+1)\Big) \frac{\ln(2\pi)}{4}  \\
  &+ \frac{1}{2}\ln(q) -\sum\limits_{i=0}^{p-1} \sum\limits_{j=0}^{q-1} \ln G((1+i)/p+j\tau/q;\tau), 
\end{align*}
and according to \eqref{G_z_p/q_formula2} this expression must be equal to 
$b_0(p \tau/q)$ -- the constant term in the asymptotic expansion of $\ln G(z;p \tau/q)$. This concludes the proof of \eqref{b0_identity3}.

Now we can prove formulas \eqref{formula_b0}. We set $p=q$  and conclude from
\eqref{b0_identity3} that 
\begin{align}
\nonumber
\label{b_0_proof2}
    b_0(\tau)&= 
    \frac{1}{q^2-1} 
    \bigg[ - q^2 a_0(\tau) \ln(\tau)+a_0(\tau) \ln(q \tau) \\&-\Big(q-1+(q-1)(q(\tau+1)+1)\Big) \frac{1}{4} \log(2\pi) -\frac{1}{2}\ln(q)
\\
\nonumber 
&\qquad +\sum\limits_{i=0}^{p-1} \sum\limits_{j=0}^{q-1} \ln G((1+i)/p+j\tau/q;\tau)\bigg].
\end{align}
Taking the limit as $q\to \infty$ we obtain the first equation in \eqref{formula_b0}. To prove the second equation in \eqref{formula_b0}
we set $q=2$ in \eqref{b_0_proof2}. The resulting term 
$$
\ln \Big[ G(1/2;\tau) G((1+\tau)/2;\tau) G(1+\tau/2;\tau) \Big]
$$
is simplified  using the fact
$$
G(1+\tau/2;\tau)=\Gamma(1/2) G(\tau/2;\tau)=\sqrt{2\pi }
G(\tau/2;\tau)
$$
and using the formula \eqref{G_z_p/q_formula3} with $p=2$, $z=\tau/2$ written in the form
\begin{equation*}
    G(\tau/2;\tau)G((1+\tau)/2;\tau)=G(1/2;\tau)G(\tau;2\tau).
 \end{equation*}

 Formula \eqref{eqn_b0_1} follows from Theorem 1 in \cite{Ferreira_Lopez}, where an asymptotic expansion for $G(z;1)$ is presented. 
\end{proof}

In our next result we present a complete 
asymptotic expansion of the error term ${\mathcal E}(z;\tau)$ from \eqref{eqn_lnG_E}. This was essentially obtained by Barnes in 1901 in part V of \cite{Barnes1901}: he  derived a complete asymptotic expansion for the symmetric double gamma function $\Gamma_2(z;\omega_1, \omega_2)$.
 Unfortunately, this asymptotic expansion of $\Gamma_2(z;\omega_1, \omega_2)$ is not stated as a theorem, rather the result is developed over a dozen or so pages of text and it uses a lot of auxiliary constants and functions, whose definitions are scattered across a very long paper. Thus, we felt it is important to present Barnes' result in a concise form and adapted to the double gamma function $G(z;\tau)$. To state this result, we denote by  
$$
{\mathcal S}_{\alpha,\beta}:=\{ z \in {\mathbb C}\; : \; z\neq0, \; \alpha\le \arg(z) \le \beta \}
$$
a sector in the complex plane bounded by half-lines with angles $\alpha$ and $\beta$.

\begin{theorem}[Barnes, \cite{Barnes1901}]
\label{thm2}
Let $\tau \in {\mathbb C} \setminus (-\infty,0]$ be fixed.  Let $\alpha,\beta$ be two numbers such that $-\pi < \alpha<\beta<\pi$ for which  
\begin{equation*}
{\mathcal S}_{\alpha,\beta} \subset 
{\mathbb C}\setminus \{x+\tau y \, : \, x,y\le 0\}
\end{equation*}
 and $\alpha, \beta \neq \arg(-\tau)$. 
  Then as $z\to \infty$ in the domain ${\mathcal S}_{\alpha,\beta}$ we have an asymptotic expansion
\begin{equation}\label{asymptotics_E}
{\mathcal E}(z;\tau) \sim \frac{1}{\tau} \sum\limits_{n\ge 1} \frac{(-1)^{n+1}  q_{n+2}(\tau)}{n(n+1)(n+2)} z^{-n}.
\end{equation}
The above asymptotic result is uniform in ${\mathcal S}_{\alpha,\beta}$.
\end{theorem}
 \begin{figure}[t!]
\centering
\includegraphics[width=0.5\textwidth]{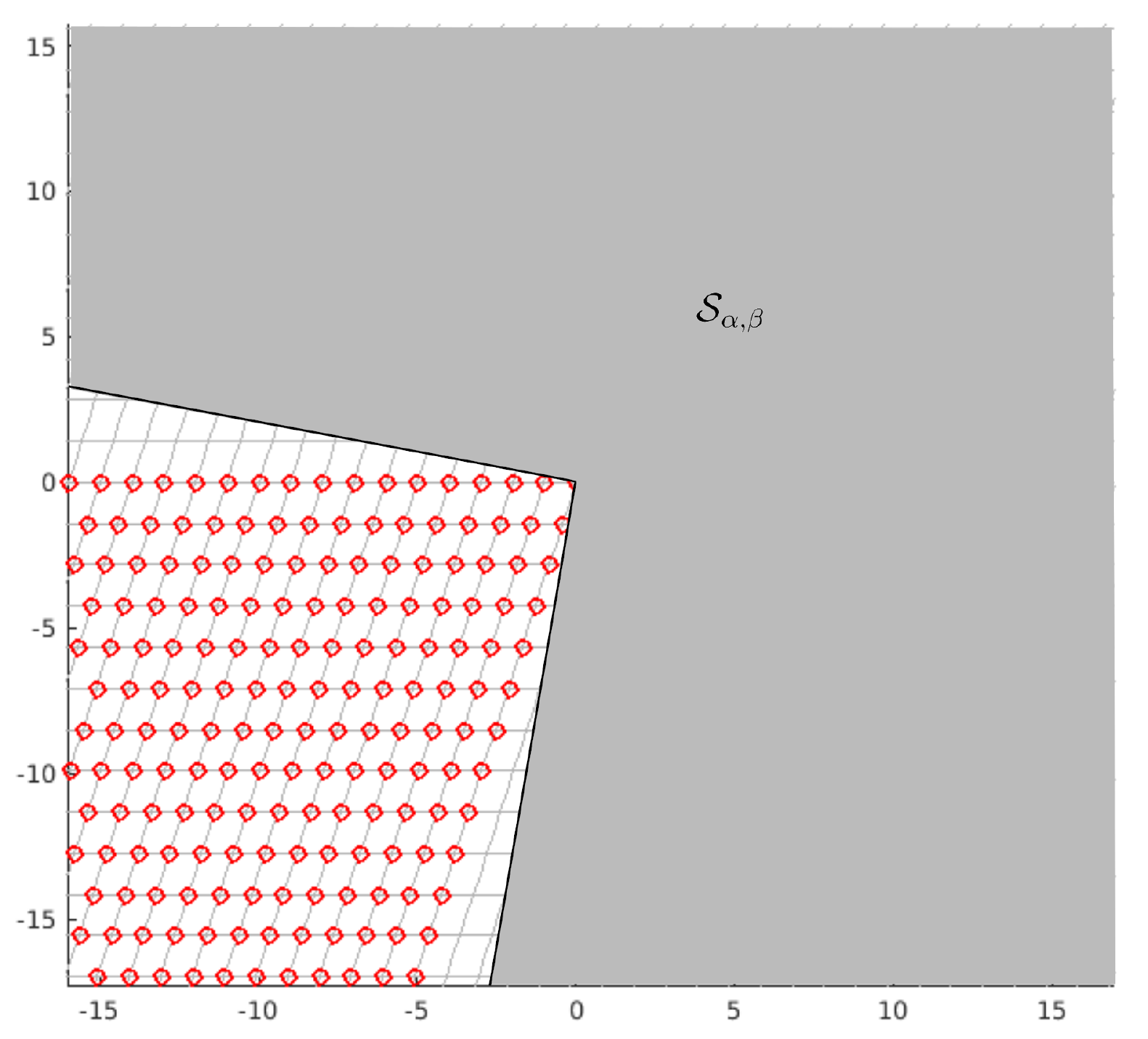}
\caption{The red circles are the zeros of $G(z;\tau)$, given by $\{m\tau+n \; : \; m,n \le 0\}$ and the shaded area shows the sector ${\mathcal S}_{\alpha,\beta} $ in which asymptotic result \eqref{asymptotics_E} holds true.  }
 \label{fig2}
\end{figure}

\begin{remark} (Added in October 2023)
The statement of the above theorem is changed here, compared with what appeared in the published version of this paper (see  \url{https://doi.org/10.1080/10652469.2023.2238115}). In the published version it was mistakenly stated that the asymptotic result \eqref{asymptotics_E} holds in a bigger domain ${\mathbb C}\setminus \{x+\tau y \, : \, x,y\le 0\}$, but it turns out that it is necessary to take a strictly smaller sector (see Figure \ref{fig2}). The problem is that the asymptotics in \eqref{asymptotics_E} would not be correct if $z$ was allowed to approach infinity along a path that comes too close to the zeros of $G(z;\tau)$, for example, along a line parallel to negative half line $(-\infty,0]$ or to the half line $\{ \tau y \; : \; y\le 0\}$. The proof that \eqref{asymptotics_E} is correct in the sector ${\mathcal S}_{\alpha,\beta}$ that satisfies the conditions of Theorem \ref{thm2} and that this asymptotic relation is indeed uniform in this sector can be obtained from formulas 
(3.13), (3.18), (3.19) in \cite{RUIJSENAARS2000107} (and formula \eqref{G_as_Gamma1} above) using Watson's lemma. 
\end{remark}

\newpage 
\begin{proof}
 In this proof we will follow the notation from Barnes' paper 
\cite{Barnes1901}. 
We set $a=0$, $\omega_1=1$ and $\omega_2=\tau$ in the formula on page 384 in \cite{Barnes1901} -- this gives us an asymptotic expansion in the form
\begin{align*}
\ln \Big[ \frac{\Gamma_2(z;1,\tau) e^{2\pi \i (m+m') {}_2S_1'(0)}}{\rho_2(1,\tau)} \Big]=
&-\frac{z^2}{2} {}_2S_1^{(3)}(0) (\ln(z)-2(m+m') \pi \i -\frac{3}{2} )\\ 
&-\frac{z}{2} {}_2S_1^{(2)}(0)(\ln(z)-2(m+m') \pi \i - 1) \\
&-{}_2S_1'(0) (\ln(z)-2(m+m') \pi \i)
+\sum\limits_{n\ge 1} \frac{(-1)^n {}_2S_{n+1}'(0)}
{n(n+1) z^n}
\end{align*}
In this formula $m$ and $m'$ are certain integers, which in our case are both equal to zero, due to our choice of $\omega_1=1$ and $\omega_2=\tau \in {\mathbb C}\setminus (-\infty,0]$), see the discussion in \textsection 22 on pages 291-293 in \cite{Barnes1901}. 
The constant $\rho_2(1,\tau)$ is defined on page 334 in \cite{Barnes1901} (see also page 354). The coefficients ${}_2S_1^{(3)}(0)$, ${}_2S_1^{(2)}(0)$ and $\{{}_2S_{n}'(0)\}_{n\ge 1}$  (which implicitly depend on $\tau$) are defined via the generating function 
\begin{equation*}
\frac{z}{(1-e^{- z})(1-e^{-\tau z})}=
\frac{{}_2S_1^{(3)}(0)}{z}-
{}_2S_1^{(2)}(0)+\frac{{}_2S_1'(0)}{1!} z+\dots
+(-1)^{n-1} \frac{{}_2S_n'(0)}{n!} z^n+\dots,
\end{equation*}
which can be found on page 283 in \cite{Barnes1901}. Comparing the above series with \eqref{qn_generating_function} we see that 
$$
{}_2S_n'(0)=\frac{q_{n+1}(\tau)}{\tau (n+1)}.
$$
Combining the above equations with \eqref{eqn_lnG_E} and formula
\begin{equation*}
G(z;\tau)=\frac{(2\pi)^{z/2} 
\tau^{(1+\frac{1}{\tau})\frac{z}{2}-\frac{z^2}{2\tau} -1}}{\Gamma_2(z;1,\tau)}, 
\end{equation*}
which follows from \eqref{G_as_Gamma1},  we obtain \eqref{asymptotics_E}. 
\end{proof}

In our next result we present some new expressions for the gamma modular forms $C(\tau)$ and $D(\tau)$.

\begin{proposition}\label{Proposition_C_D_tau}
${}$
\begin{itemize}
\item[(i)] Assume that $\re(\tau)>0$. Then 
\begin{align}\label{eqn_C_integral}
    C(\tau)&=  \frac{1}{2\tau} \ln(2\pi) \\
    \nonumber &- 
    \int_0^{\infty} \bigg( 
    \frac{e^{-\tau x}}{(1-e^{-x})(1-e^{-\tau x})} 
    -\frac{e^{-x}}{\tau x} \Big( \frac{1}{1-e^{-x}}+1-\frac{\tau}{2} \Big) \bigg) d x
  ,\\ 
  \label{eqn_D_integral}
    D(\tau)&= \int_0^{\infty} \bigg( \frac{x e^{-\tau x}}{(1-e^{-x})(1-e^{-\tau x})}
    -\frac{e^{-x}}{\tau x} \bigg) d x.
\end{align}
\item[(ii)] For $\tau \in {\mathbb C} \setminus (-\infty, 0]$ we have
\begin{align}
\label{eqn_C_formula_G}
 C(\tau)&=-\frac{\tau-1}{2\tau} \ln(\tau)+\frac{1}{2} \ln(2\pi)-\frac{d}{d z} \ln G(z;\tau) \Big \vert_{z=\tau},\\
 \label{eqn_D_formula_G}
 D(\tau)&=-\frac{\ln(\tau)}{\tau}-\frac{d^2}{d z^2} \ln G(z;\tau) \Big \vert_{z=\tau}.
 \end{align}
\end{itemize}
\end{proposition}
\begin{proof}
To prove formula \eqref{eqn_C_integral}, we recall the following two integral representations (see formulas 8.341.7 and 8.361.1 in \cite{Jeffrey2007}), valid for $\re(z)>0$
\begin{align*}
    \ln \Gamma(z)&=\int_0^{\infty} \Big( (z-1) e^{-x}+ \frac{e^{-zx}-e^{-x}}{1-e^{-x}} 
    \Big) \frac{dx}{x}, \\
    \psi(z)&= \int_0^{\infty} \Big( \frac{e^{-x}}{x}-\frac{e^{-zx}}{1-e^{-x}} \Big) d x.
\end{align*}
Using these formulas and \eqref{def_C} (and the fact that $\psi^{(j)}(m\tau)=O(m^{-j})$) we obtain
\begin{align*}
C(\tau)&=\frac{1}{2\tau} \ln(2\pi)+
\int_0^{\infty} \bigg[\Big( m-\frac{1}{2}\Big) \frac{e^{-x}}{x}
-\frac{e^{-\tau x}-e^{-m \tau x}}{(1-e^{-\tau x})(1-e^{-x})}
\\
&-\frac{1}{2} \frac{e^{-m \tau x}}{1-e^{-x}}
-\Big(m-\frac{1}{\tau}\Big) \frac{e^{-x}}{x}-
\frac{e^{-m\tau x}-e^{-x}}{\tau x (1-e^{-x})} \bigg] d x+O(m^{-1})\\
&=\frac{1}{2\tau} \ln(2\pi)+
\int_0^{\infty} \bigg\{ \Big( \frac{1}{\tau}-\frac{1}{2} \Big) \frac{e^{-x}}{x} 
-\frac{e^{-\tau x}}{(1-e^{-\tau x})(1-e^{-x})}+
\frac{e^{-x}}{\tau x (1-e^{-x})} \bigg\} d x\\
&+\int_0^{\infty} \bigg\{ \frac{1}{(1-e^{-\tau x})(1-e^{-x})}-
\frac{1}{2(1-e^{-x})}- \frac{1}{\tau x (1-e^{-x})} \bigg\} e^{-m \tau x} d x
+O(m^{-1}). 
\end{align*}
One can check that the two functions in braces are bounded as $x \to 0$: this justifies the fact that we could rearrange the integral as is shown above and it also implies that the second integral in the above representation of $C(\tau)$ converges to zero as $m\to +\infty$. Thus in the limit $m\to +\infty$ the above formula gives us the desired result \eqref{eqn_C_integral}. 

Let us now establish \eqref{eqn_C_formula_G}. Assume that $\re(\tau)>0$. From \eqref{lnG_integral} we find 
$$
\frac{d}{dz} \ln G(z;\tau)\Big \vert_{z=\tau}=
\int_0^{\infty}      \bigg( \frac{x e^{-\tau  x}}{(1-e^{-x})(1-e^{-\tau x})}- \frac{e^{-\tau x}}{1-e^{-\tau x}}
     - \frac{1}{2\tau} e^{-\tau x} \bigg) \frac{dx}{x},
$$
which implies 
\begin{align}
\label{eqn_C_I}
&C(\tau) + \frac{d}{dz} \ln G(z;\tau)\Big \vert_{z=\tau}=
\frac{1}{2\tau} \ln(2\pi) + I  
\\
\nonumber
&=\frac{1}{2\tau} \ln(2\pi)+\int_0^{\infty} \bigg( \frac{e^{-x}}{\tau x} \Big( \frac{1}{1-e^{-x}}+1-\frac{\tau}{2} \Big) 
- \frac{e^{-\tau x}}{x(1-e^{-\tau x})}
     - \frac{1}{2\tau x} e^{-\tau x}\bigg) dx.
\end{align}
We separate the integral $I$ in the above formula as follows: 
\begin{align}\label{eqn_I_I1234}
I&=I_1-I_2+I_3+I_4\\
\nonumber
&=\int_0^{\infty} \Big(\frac{1}{1-e^{-x}}-\frac{1}{2}-\frac{1}{x} \Big) \frac{e^{-x}}{\tau x} d x- \int_0^{\infty} \Big(\frac{1}{1-e^{-\tau x}}-\frac{1}{2}-\frac{1}{\tau x} \Big) \frac{e^{- \tau x}}{x} d x\\
\nonumber
&+\int_0^{\infty} \frac{1}{\tau} \Big(e^{-x}(1+x(1-\tau))-e^{-\tau x}  \Big) x^{-2} d x+
\int_0^{\infty} \frac{1+\tau}{2 \tau}  \Big( e^{-x}-e^{-\tau x}\Big) x^{-1} d x.
\end{align}
To compute $I_1$ we evaluate for $\re(w)>2$
\begin{align*}
f_1(w)&= \int_0^{\infty} \Big(\frac{1}{1-e^{-x}}-\frac{1}{2}-\frac{1}{x} \Big) e^{-x} x^{w-1} d x\\
&=\int_0^{\infty} \frac{x^{w-1} d x}{e^{x} - 1}-
\frac{1}{2}\int_0^{\infty} e^{-x} x^{w-1} d x -\int_0^{\infty} e^{-x} x^{w-2} d x\\
&=\Gamma(w) \zeta(w)-\frac{1}{2} \Gamma(w) - \Gamma(w-1)\\
&=\Gamma(w+1) \Big( \frac{\zeta(w)-\frac{1}{2}}w - \frac{1}{w(w-1)} \Big),
\end{align*}
where we have used well known integral representation formulas for the Riemann zeta function
and for the Gamma function. The function in the right-hand side is analytic in $\re(w)>-1$ and it is clear that $(1/\tau) f_1(w) \to I_1$ as $w \to 0^+$. To compute this limit we use the fact that 
$$
\zeta(w)=-\frac{1}{2}-\frac{w}{2} \ln(2\pi) + O(w^2), \;\;\; w\to 0
$$
and we obtain 
\begin{equation}
    I_1=\frac{1}{\tau} \Big(1-\frac{1}{2} \ln(2\pi)\Big). 
\end{equation}
A change of variables $x\mapsto x/\tau$ gives us an identity $I_2=\tau I_1$. The third integral is computed via the same trick as above: we evaluate for $\re(w)>0$
\begin{align*}
f_3(w)&=\int_0^{\infty} \Big(e^{-x}(1+x(1-\tau))-e^{-\tau x}  \Big) x^{w-1} d x \\
&=
\Gamma(w)+(1-\tau) \Gamma(w+1)-\tau^{-w} \Gamma(w)\\
&=\Gamma(w+2) \Big( \frac{1-\tau^{-w}}{w(w+1)}+\frac{1-\tau}{w+1} \Big),
\end{align*}
observe that this function is analytic for $\re(w)>-2$
and taking the limit as $w\to -1$ we obtain 
$$
I_3=\frac{1}{\tau} \lim\limits_{w\to -1} f_3(w)=\frac{1}{\tau} (\tau-1-\tau \ln(\tau)). 
$$
Finally, with the help of Frullani integral \cite{Reyna} we evaluate 
$$
I_4=\frac{1+\tau}{2 \tau} \ln(\tau).
$$
Combining the above expressions for $I_1$, $I_2$, $I_3$ and $I_4$  we have 
$$
I=I_1-I_2+I_3+I_4=\frac{\tau-1}{2\tau} \ln\Big(\frac{2\pi}{ \tau}\Big), 
$$
and this formula together with \eqref{eqn_C_I} gives us \eqref{eqn_C_formula_G} for $\re(\tau)>0$. The fact that this result holds in the larger domain $\tau \in {\mathbb C} \setminus (-\infty,0]$ follows by analytic continuation.

The proof of \eqref{eqn_D_integral} and \eqref{eqn_D_formula_G} follows exactly the same strategy as the proof of \eqref{eqn_C_integral} and \eqref{eqn_C_formula_G} and is left to the reader. 
\end{proof}

At page 366 in \cite{Barnes1899} Barnes has established a reflection formula for $D(\tau)$. Let $K=K(k)$ ($E=E(k)$) be the complete elliptic integral of the first (respectively, second) kind and
as usual we denote $K'=K(k')$, where $k'=\sqrt{1-k^2}$. Assume that the imaginary part of $\tau=\i K'/K$ is positive. Then the following identity is true
\begin{equation}
    \label{eqn:D_tau_D_minus_tau}
D(\tau)+D(-\tau)=\frac{\pi^2}{6}\Big(\frac{1}{\tau^2}-1\Big)-\frac{\pi \i}{\tau} 
+2 E K- 
\frac{2}{3} K^2 (1+k'^2).
\end{equation}
In the above formula we have corrected a typo from \cite{Barnes1899}: Barnes had a term  $\pi^2/12$ instead of $\pi^2/6$ in the above formula. The mistake appears in the third line of the equation on top of page 365 in \cite{Barnes1899}: both of the infinite sums appearing on that line should be multiplied by a factor of two. We have also verified numerically that the identity
\eqref{eqn:D_tau_D_minus_tau} is correct. It would be interesting to establish a similar reflection formula for $C(\tau)$.


\section{Proofs of identities 
\eqref{eq:G_1_over_tau}, \eqref{G_z_p/q_formula2}
and \eqref{eq:modular}}\label{section_proofs}


 \begin{figure}[t!]
\centering
\includegraphics[width=0.5\textwidth]{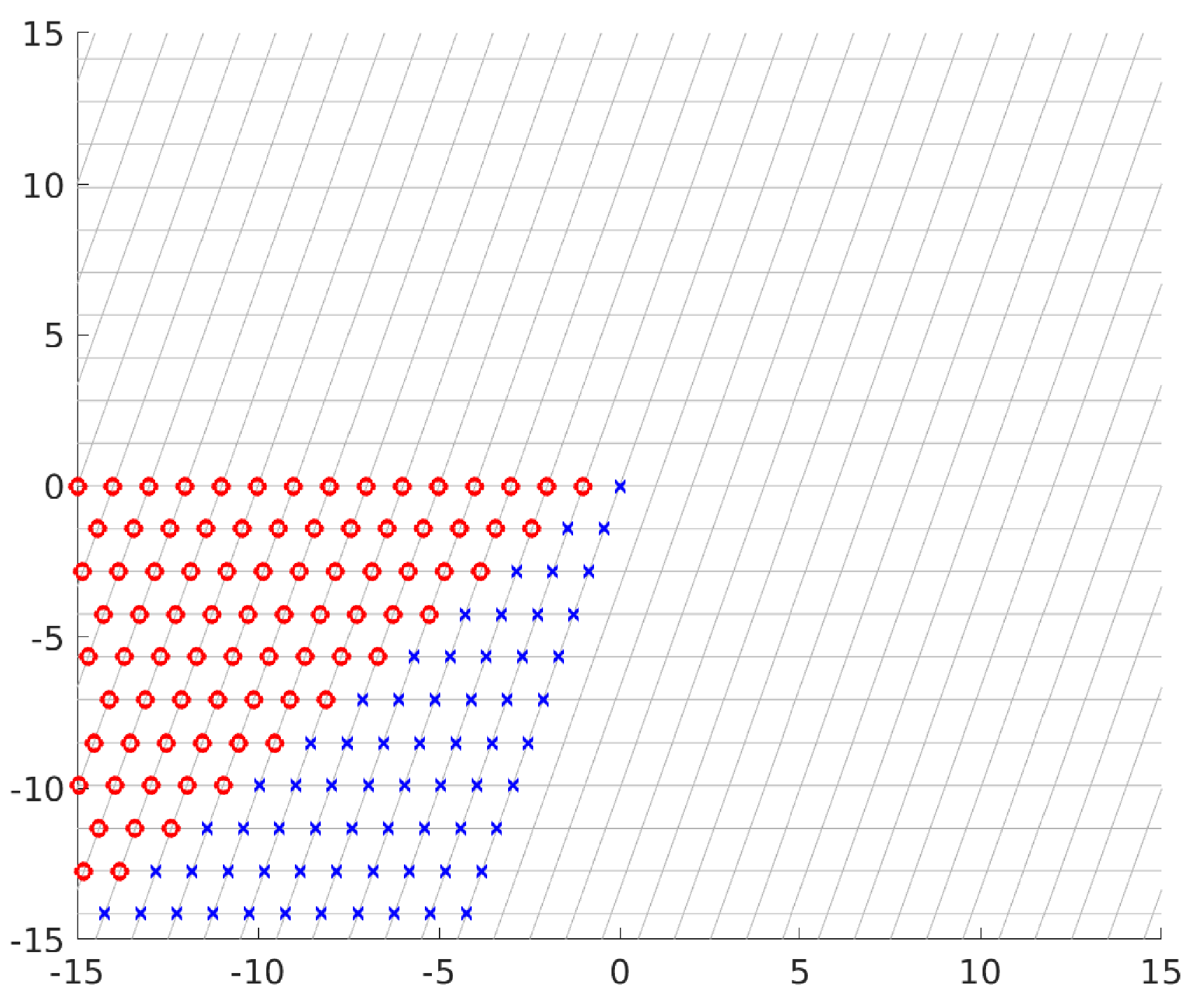}
\caption{The red circles (blue crosses) are the zeros of $G(z+1;1+\tau)$ (respectively, $G(z/\tau;1+1/\tau)$).}
 \label{fig1}
\end{figure}

\noindent 
{\it Proof of the product identity \eqref{eq:modular}.}
We consider a function
 $$
 f(z)=\frac{ G(z;\tau)}{G(z+1;1+\tau)G(z/\tau;1+1/\tau)}. 
 $$
 It is easy to check that this function is entire and zero-free: the zeros of the denominator cover the same quadrant of the lattice 
 $\{m\tau+n \; : \; m,n\le 0\}$ as the zeros of the numerator, see Figure \ref{fig1}. Since the Barnes $G$-function is an entire function of order two (this follows from the Weierstrass product representation \eqref{Weierstrass_G}), we conclude that $f(z)=\exp(A+Bz+Cz^2)$ for some constants $A,B,C$. Our aim is to prove that 
 \begin{equation}
     \label{eqn:proposed_f}
 f(z)=\Big(\frac{1+\tau}{\tau}\Big)^{\frac{z^2}{2\tau}-(1+\tau) \frac{z}{2\tau}+1}
  (2\pi)^{-\frac{z}{2\tau}}.
 \end{equation}

 There are several methods that could be used to determine the three constants $A,B,C$. The first method is based on computing the asymptotics of $\ln f(z)$ as $z\to \infty$ via \eqref{eqn_lnG_E}. The second method would be to choose $A$, $B$ and $C$ so that the function 
 $$
 g(z)=e^{A+Bz+Cz^2} G(z+1;1+\tau)G(z/\tau;1+1/\tau)
 $$
 satisfies the same functional equations \eqref{eq:funct_rel_G} as $G(z;\tau)$ and to enforce the normalization condition $g(1)=1$. Here we will follow the third method: we will compute the values of $f(0)$, $f(\tau)$ and $f(1+\tau)$ and will then check that the function given in 
\eqref{eqn:proposed_f} agrees with these values. 
 
 First, we note that  
 $G(z;\tau)=(z/\tau)(1+o(1))$ as $z\to 0$ (this follows from \eqref{Weierstrass_G}), thus we obtain 
 $f(0)=(1+\tau)/\tau$. 
Formula \eqref{eqn:G(tau;tau)} helps us compute
 $$
 f(\tau)=\Big(\frac{1+\tau}{\tau} \Big)^{\frac{1}{2}} (2\pi)^{-\frac{1}{2}}.
 $$
 Similarly, using the first formula in \eqref{eq:funct_rel_G} we find $G(1+\tau;\tau)=G(\tau;\tau)$ and we conclude that 
 $$
 f(1+\tau)=\Big(\frac{1+\tau}{\tau} \Big) (2\pi)^{-\frac{1}{2\tau}-\frac{1}{2}}. 
 $$
Now one can check that the function in the right-hand side of \eqref{eqn:proposed_f}  has the same values at  $f(0)$, $f(\tau)$ and $f(1+\tau)$ computed above. Since a polynomial of degree two is uniquely determined by its values at three distinct points, this ends the proof of 
\eqref{eq:modular}. 
\qed

\vspace{0.5cm}
\noindent 
{\it Proof of the modular transformation \eqref{eq:G_1_over_tau}.} Let $\tau \in {\mathbb C}\setminus (-\infty,0]$ be fixed. Consider the entire function
$$
f(z)=(2\pi)^{\frac{z}2 \left(1-\frac1{\tau} \right)} \tau^{\frac{z-z^2}{2\tau}+\frac{z}2-1}  G\left(\frac{z}{\tau};\frac{1}{\tau}\right), \;\;\; z\in {\mathbb C}.
$$
We will check that $f$ satisfies both functional equations \eqref{eq:funct_rel_G} and the normalization condition 
$f(1)=1$: this will imply that $f(z)=G(z;\tau)$.

Using the second identity in \eqref{eq:funct_rel_G} we obtain
\begin{align*}
    f(z+1)&=(2\pi)^{\frac{z+1}2 \left(1-\frac1{\tau} \right)} \tau^{\frac{z+1-(z+1)^2}{2\tau}+\frac{z+1}2-1}  G\left(\frac{z}{\tau}+\frac{1}{\tau};\frac{1}{\tau}\right)\\
    &=(2\pi)^{\frac{z+1}2 \left(1-\frac1{\tau} \right)} \tau^{\frac{z+1-(z+1)^2}{2\tau}+\frac{z+1}2-1}  (2\pi)^{\frac{1}{2\tau}-\frac{1}{2}} \tau^{\frac{z}{\tau}-\frac{1}{2}} \Gamma \Big(\frac{z}{\tau}\Big)G\left(\frac{z}{\tau};\frac{1}{\tau}\right)\\
    &= \Gamma \Big(\frac{z}{\tau}\Big)(2\pi)^{\frac{z}2 \left(1-\frac1{\tau} \right)} \tau^{\frac{z-z^2}{2\tau}+\frac{z}2-1}  G\left(\frac{z}{\tau};\frac{1}{\tau}\right)
    = \Gamma \Big(\frac{z}{\tau}\Big) f(z). 
\end{align*}
Similarly, using the first identity in \eqref{eq:funct_rel_G} we find
\begin{align*}
    f(z+\tau)&=(2\pi)^{\frac{z+\tau}2 \left(1-\frac1{\tau} \right)} \tau^{\frac{z+\tau-(z+\tau)^2}{2\tau}+\frac{z+\tau}2-1}  G\left(\frac{z}{\tau}+1;\frac{1}{\tau}\right)\\
    &=(2\pi)^{\frac{z+\tau}2 \left(1-\frac1{\tau} \right)} \tau^{\frac{z+\tau-(z+\tau)^2}{2\tau}+\frac{z+\tau}2-1}
    \Gamma(z)G\left(\frac{z}{\tau};\frac{1}{\tau}\right)\\
    &= (2\pi)^{\frac{\tau-1}{2}} \tau^{-z+\frac{1}{2}}\Gamma(z)  (2\pi)^{\frac{z}2 \left(1-\frac1{\tau} \right)} \tau^{\frac{z-z^2}{2\tau}+\frac{z}2-1}  G\left(\frac{z}{\tau};\frac{1}{\tau}\right)
    \\
    &= (2\pi)^{\frac{\tau-1}{2}} \tau^{-z+\frac{1}{2}}\Gamma(z)   f(z). 
\end{align*}
 Since $f$ satisfies the two functional equations 
\eqref{eq:funct_rel_G} and the normalization condition $f(1)=1$ (which follows from \eqref{eqn:G(tau;tau)}), we conclude that $f(z)=G(z;\tau)$.
\qed

\vspace{0.5cm}
\noindent 
{\it Proof of the multiplication formula \eqref{G_z_p/q_formula2}.}
We will follow a similar path as in the above proof of 
\eqref{eq:modular}. Let $\tau \in {\mathbb C}\setminus (-\infty,0]$ and $p,q \in {\mathbb N}$ be fixed. Consider the entire function
$$
h(z)=\prod\limits_{i=0}^{p-1} \prod\limits_{j=0}^{q-1} 
 \frac{G((z+i)/p+j \tau/q;\tau)}{G((1+i)/p+j \tau/q;\tau)}.
$$
The zeros of $G((z+i)/p+j \tau/q;\tau)$ are 
$z=-(j+qn)\frac{p \tau}{q}-(i+pn)$,  $m,n \ge 0$. The union of these sets as $i$ ranges from $0$ to $p-1$ and $j$ ranges from $0$ to $q-1$ is precisely the lattice $\{-m \tau - n \; : \; m,n\ge0 \}$, thus  the function $H(z)=G(z;p\tau/q)/h(z)$ has no zeros. It is also an entire function of second order, which satisfies $H(1)=1$, thus we can write 
$H(z)=\exp(B(z-1)+C(z-1)^2)$ for some constants $B$ and $C$. To find constants $B$ and $C$ we will compute $H(0)$ and $H(p \tau/q)$. 

As $z\to 0$ we have $G(z;p \tau/q)=z q/(p\tau)(1+o(1))$ and 
\begin{align*}
h(z)&=G(z/p;\tau) \times \prod\limits_{0\le i \le p-1} \prod\limits_{\substack{0\le j \le q-1 \\ i+j>0}} 
 \frac{G((z+i)/p+j \tau/q;\tau)}{G((1+i)/p+j \tau/q;\tau)}\\
 &=
z/(p\tau) (1+o(1))\prod\limits_{j=1}^{q-1} 
 \frac{G(j \tau/q;\tau)}{G(1+j \tau/q;\tau)}\\
 &=z/(p\tau) (1+o(1))\prod\limits_{j=1}^{q-1} \Gamma(j/q)^{-1}=z/(p\tau) (1+o(1)) 
 \times (2\pi)^{\frac{1-q}{2}} q^{\frac{1}{2}}.
\end{align*}
Here in the third step we used the first functional equation from \eqref{eq:funct_rel_G} and
in the last step we used the multiplication formula for the gamma function. The above results imply $H(0)=(2\pi)^{(q-1)/2} q^{1/2}$. 

Next we set $z=p \tau/q$ and compute using \eqref{eqn:G(tau;tau)}
$$
G(p \tau/q;p \tau/q)=(2\pi)^{\frac{p \tau-q}{2 q}} (p \tau/q)^{-\frac{1}{2}}. 
$$
Using the same method as above we also compute 
\begin{align*}
h( p \tau/q)&= \prod\limits_{i=0}^{p-1} \prod\limits_{j=0}^{q-1} 
 \frac{G(i/p+(j+1) \tau/q;\tau)}{G((1+i)/p+j \tau/q;\tau)}\\&=G(\tau;\tau) \prod\limits_{i=1}^{p-1} \frac{G(i/p+\tau; \tau)}{G(i/p;\tau)} \times 
 \prod\limits_{j=1}^{q-1} \frac{G(j \tau/q;\tau)}{G(1+j \tau/q;\tau)}\\
 &=(2\pi)^{(\tau-1)/2} \tau^{-1/2} \prod\limits_{i=1}^{p-1} (2\pi)^{(\tau-1)/2} \tau^{-i/p+1/2} \Gamma(i/p) \times 
 \prod\limits_{j=1}^{q-1} \Gamma(j /q)^{-1}\\&=
 (2\pi)^{(p\tau-q)/2} (p\tau/q)^{-1/2}
\end{align*}
and we conclude that $H(p\tau/q)=(2\pi)^{(p\tau-q)(1-q)/(2q)}$. Using these two values of $H(0)$ and $H(p \tau/q)$ and the fact that $H(1)=1$ we conclude that 
$$
H(z)=q^{\frac{1}{2 p \tau}(z-1)(qz-p \tau) }  (2\pi)^{-\frac{q-1}{2}(z-1)},
$$
which ends the proof of \eqref{G_z_p/q_formula2}. 
\qed

\section*{Acknowledgements}
Research  was supported by the Natural Sciences and Engineering Research Council of Canada. The authors would like to thank two anonymous referees for careful reading of the paper and for providing constructive comments and suggestions.

\section*{Disclosure statement}
No potential conflict of interest was reported by the authors.




\appendix
\section{Polynomials $q_n(\tau)$  for $0\le n \le 21$}\label{AppendixA}


The odd-numbered polynomials $q_n$ with $n=1,3,\dots,21$ are
\begin{alignat*}{4}
&q_1(\tau)=-\frac{1}{2}(1+\tau), \qquad \qquad \;\;\;  &&q_3(\tau)=-\frac{1}{4}\tau(1+\tau),\\
&q_5(\tau)=\frac{1}{12}\tau(1+\tau^3), &&q_7(\tau)=-\frac{1}{12}\tau(1+\tau^5), \\& q_9(\tau)=\frac{3}{20}\tau(1+\tau^7),
&&q_{11}(\tau)=-\frac{5}{12}\tau(1+\tau^9), \\
&q_{13}(\tau)=\frac{691}{420} \tau (1+\tau^{11}), && q_{15}(\tau)=-\frac{35}{4} \tau(1+\tau^{13}),\\  & q_{17}(\tau)=\frac{3617}{60} \tau(1+\tau^{15}),  && q_{19}(\tau)=-\frac{43867}{84} \tau(1+\tau^{17}), \\
& q_{21}(\tau)=\frac{1222277}{220} \tau(1+\tau^{19}).  &&
\end{alignat*}

The even-numbered polynomials $q_n$ with $n=0,2,\dots,20$ are
\begin{align*}
q_0(\tau)&=1, \\
q_2(\tau)&=\frac{1}{6}(1+3\tau+\tau^2),\\
q_4(\tau)&=-\frac{1}{30}(1-5\tau^2+\tau^4), \\
q_6(\tau)&=\frac{1}{84}(2-7 \tau^2-7 \tau^4+2\tau^6), \\
q_8(\tau)&=-\frac{1}{90}(3-10 \tau^2-7 \tau^4-10\tau^6+3\tau^8), \\
q_{10}(\tau)&=\frac{1}{132}(10-33\tau^2-22 \tau^4-22 \tau^6-33 \tau^8+10 \tau^{10}), \\
q_{12}(\tau)&=-\frac{1}{5460} (1382 - 4550\tau^2 - 3003 \tau^4 - 2860 \tau^6 - 3003 \tau^8 - 4550 \tau^{10} + 1382 \tau^{12}), \\
q_{14}(\tau)&=\frac{1}{180} 
(210 - 691 \tau^2 - 455 \tau^4 - 429 \tau^6 - 429 \tau^8 - 455 \tau^{10} - 691 \tau^{12} + 210 \tau^{14}),\\
q_{16}(\tau)&=-\frac{1}{1530} 
(10851 - 35700\tau^2 - 23494\tau^4 - 22100\tau^6 - 21879\tau^8 - 22100\tau^{10} \\
&\qquad \qquad \qquad \qquad \qquad \qquad  - 23494\tau^{12} - 35700\tau^{14}  
+10851\tau^{16}),\\
q_{18}(\tau)&=\frac{1}{7980} 
(438670 - 1443183\tau^2 - 949620\tau^4 - 892772\tau^6 - 881790\tau^8 - 881790\tau^{10} \\
&\qquad \qquad \qquad - 892772\tau^{12}  - 949620\tau^{14} - 1443183\tau^{16}+438670\tau^{18}),\\
q_{20}(\tau)&=-\frac{1}{13860} 
(7333662 - 24126850\tau^2 - 15875013\tau^4 - 14922600\tau^6 - 14730738\tau^8 \\  & \qquad \qquad \qquad
 - 14696500\tau^{10} - 14730738\tau^{12} - 14922600\tau^{14}- 15875013\tau^{16} \\
&  \qquad \qquad \qquad  \qquad \qquad - 24126850\tau^{18} + 7333662\tau^{20}).
\end{align*}

\end{document}